\documentclass[openany, a4paper, 12pt]{article}

 \usepackage{mathrsfs,amsfonts,amsmath}
 \usepackage{color}

 \makeatletter
 \@addtoreset{equation}{section}
 \makeatother
 \newtheorem{theorem}{Theorem}[section]
 \newtheorem{definition}{Definition}[section]
 \newtheorem{hypothesis}{Hypothesis}[section]
 \newtheorem{lemma}{Lemma}[section]
 \newtheorem{proposition}{Proposition}[section]
 \newtheorem{corollary}{Corollary}[section]
 
 \newtheorem{example}{Example}[section]

 \newtheorem{Theo}{Theorem}[section]
 \newtheorem{Lemma}[Theo]{Lemma}

 \newtheorem{Rem}[Theo]{Remark}

 \def\bdefinition{\begin{definition}\sl{}\def\edefinition{\end{definition}}}

 \def\beqlb{\begin{eqnarray}}\def\eeqlb{\end{eqnarray}}
 \def\beqnn{\begin{eqnarray*}}\def\eeqnn{\end{eqnarray*}}

 \def\ar{\!\!&}

 \def\mbb{\mathbb}

 \def\proof{\noindent{\it Proof.~~}}
 \def\qed{\hfill$\Box$\medskip}

\newcommand{\bcen}{\begin{center}}
\newcommand{\ecen}{\end{center}}
\newcommand{\bgeqn}{\begin{equation}}
\newcommand{\edeqn}{\end{equation}}

\newcommand{\goto}{\rightarrow}

\newcommand{\bP}{\mathbb{P}}
\newcommand{\bR}{\mathbb{R}}

\newcommand{\de}{\delta}

\newcommand{\ga}{\gamma}

\def\dz{\delta}

\def\ez{\epsilon}

\def\l{\underline{l}}

\def\ra{\rangle}
\def\la{\langle}
\def\r{\right}
\def\l{\left}

\begin{document}

\title{\LARGE\textbf{An integral test on time dependent local
extinction for super-coalescing Brownian motion with Lebesgue
initial measure\footnote{Supported by NSFC(10721091,11071021) and NSERC
grants.}}}

\maketitle

\bigskip

\centerline{Hui He, Zenghu Li and Xiaowen Zhou}

{\narrower{\narrower{\narrower

\begin{abstract}
This paper concerns the almost sure time dependent local extinction
behavior for super-coalescing Brownian motion $X$ with
$(1+\beta)$-stable branching and Lebesgue initial measure on $\bR$.
We first give a representation of $X$ using excursions of a
continuous state branching process and Arratia's coalescing Brownian
flow. For any nonnegative,  nondecreasing and right continuous
function $g$, put
 \[
\tau:=\sup \{t\geq 0: X_t([-g(t),g(t)])>0 \}.
 \]
We prove that $\bP\{\tau=\infty\}=0$ or $1$ according as { whether} the
integral $\int_1^\infty g(t)t^{-1-1/\beta} dt$ is finite or infinite.
\end{abstract}

\bigskip

\bigskip


\bigskip


\par}\par}\par}

\section{Introduction}

{ By a {\it super-coalescing Brownian motion} (SCBM in short) we mean} a
measure-valued stochastic process describing the time-space-mass evolution of
a particle system in $\bR$. In such a system the particles move according to
(instantaneous) coalescing Brownian motions and the masses of those particles
evolve according to independent continuous state  branching processes (CSBPs
in short) with $(1+\beta)$-stable branching law. Whenever two particles are
in the same location their masses are added up with total mass { continuing
with} the independent $(1+\beta)$-stable branching. Note that this scheme is
{ well} defined due to the additivity of the CSBP. For coalescence to happen
{ with a positive probability} we only consider SCBM on $\bR$.

The SCBM has been studied in \cite{DLZ04, LWX04, Zho07}. With
arbitrary finite initial measure it can be obtained by taking a
high-density/small-particle limit of the empirical measure process
of coalescing-branching particle system with Poisson initial
measure. Its probability law can be specified by the duality on
coalescing Brownian motions.

Formally, the SCBM with Radon initial measure $\mu$ can be constructed by
taking a monotone limit of SCBMs with initial measures $\mu$ truncated over
increasing finite intervals. In this paper we will present  a direct
construction of the SCBM using excursions of the CSBP and Arratia's
coalescing Brownian flow following  Dawson et al. \cite{DLZ04}. A similar
construction was { first} proposed in Dawson and Li \cite{DL03} for
superprocess with dependent spatial motion. This procedure allows us to
construct the SCBM simultaneously for all { times} $t\geq 0$.  It turns out
to be handy for the later coupling arguments in proving our main results.

Almost sure {\it local extinction} for super-Brownian motion on $\bR^n$ says
that, given any bounded Borel set in $\bR^n$, almost surely the
measure-valued process does not charge on it after a time long enough. It
often occurs in low dimensions. For super-Brownian motion with Lebesgue
initial measure it was first studied in Iscoe \cite{Isc88} via analyzing the
super-Brownian occupation time using nonlinear PDE { arising from the Laplace
functional}. By time dependent local extinction we mean the local extinction
behavior where the size of the above-mentioned set also depends on $t$.
Almost sure time dependent local extinction was discussed in  Fleischmann et
al \cite{FKX06}. An integral test was found in Zhou \cite{Zho08} on the
almost sure time dependent local extinction for super-Brownian motion with
$(1+\beta)$-stable branching and Lebesgue initial measure. Its proof is a
Borel-Cantelli argument based on estimates of extinction probabilities. The
additivity for super-Brownian motion plays a crucial role in the proof there
since it allows us to decompose one super-Brownian motion into independent
super-Brownian motions for different purposes and then treat them separately.

The SCBM often shares similar asymptotic properties as super-Brownian
motion. In this paper we are going to show that the same integral test in
Zhou \cite{Zho08} is also valid for the SCBM on $\bR$. More precisely,
for any nonnegative, nondecreasing  and right continuous function $g$ on
$\bR^+$, we are going to show that the probability of seeing any mass
over interval $[-g(t), g(t)]$ for time $t$ large enough is either $0$ or
$1$ depending on whether the integral $\int_1^\infty g(t)
t^{-1-1/\beta}dt$ is finite or not.

 It is evident that  the integral test remains the same for the
superprocess with trivial spatial motion. This should not come as a
surprise in view of the law of iterated logarithm of the Brownian
motion. Unfortunately, we were unable to reduce the proof of the
result for the SCBM to the trivial spatial motion case using the law
of iterated logarithm. Instead we give a direct proof.

The main difficulty of the direct proof is that the SCBM is no
longer additive due to the dependence of coalescing spatial motion.
As a result we have to adopt strategies that are quite different
from Zhou \cite{Zho08} to tackle this problem.  The excursion
representation plays an important role in our proof. By the
excursion representation, one SCBM can be represented as sum of
several SCBMs starting from disjoint intervals. The key is that
those SCBMs are not independent. So in one direction of our proof,
we will use a coupling argument by introducing two coalescing
Brownian systems and comparing the asymptotical behaviors of the two
systems since the Borel-Cantelli lemma requires that the events are
independent. More details will be given in the following.

 For { the case} $\int_1^\infty g(t)t^{-1-1/\beta}dt<\infty$ we first choose a
sequence of times $(t_n)$ increasing geometrically. Then
for each $n$ we decompose the SCBM with Lebesgue
initial measure into two (dependent) SCBMs $J^n$ and $K^n$ starting
from Lebesgue measures restricted to interval  $[-2g(t_n), 2g(t_n)]$
 and its complement, respectively. We can show that, for large $n$,
both the probability for $J_t^n$ to survive up to time $t_n$ and the
probability for $K^n_t$ to ever charge interval $[-g(t_n),g(t_n)]$
before time $t_n$  are small enough. Consequently, the almost sure
local extinction with respect to $g$ occurs following a
Borel-Cantelli argument.

In the other direction of our proof, when $g$ increases fast enough we first
choose a sequence of  times $(t_n)$ strictly increasing to $\infty$ and the
associated disjoint intervals $[l_n,r_n], n=1,2,\ldots$ in $\bR^+$. We then
consider an SCBM $\bar{X}$ starting from Lebesgue measure restricted to the
region ${ (\cup_{i=1}^{\infty}[-r_i, -l_i])\cup (\cup_{j=1}^\infty
[l_j,r_j])}$. We are  able to choose the spacings between intervals properly
to satisfy the following constrains. On one hand, the spacing is not too
small so that for each $n$, up to time $t_n$ the mass started from interval
$[-r_n,-l_n]\cup [l_n, r_n]$ at time $0$ is very unlikely to interact with
masses  initiated from { the} other intervals. On the other hand, the spacing
is also not too large so that the process $\bar{X}$ still has enough initial
mass to start with and by time $t_n$ the probability
$\bP\{\bar{X}_{t_n}([-g({t_n}),g({t_n})])>0\}$ is not too small. Then the
proof can be carried out by coupling arguments together with several
Borel-Cantelli arguments.

The approaches developed in this paper can be modified to study the
almost sure time dependent local extinction for SCBM with L\'evy
branching mechanism other than stable branching. But we do not expect the
result to be as clean.

The rest of the paper is arranged as follows. In Section 2 we
present the construction of SCBM using Arratia's flow and the
branching excursion law.  Our main results of integral tests on
almost sure local extinction, Theorem \ref{main-finite} and Theorem
\ref{Main} and their proofs are presented in Section 3.

\section{A construction of SCBM with branching excursions
and Arratia's flow}

Let $\gamma\ge0$ and $0<\beta\leq1$ be fixed constants. A {\it
continuous state critical branching process} (CSBP) with
$(1+\beta)$-stable branching is a right continuous  strong Markov
process taking values in $[0,\infty)$ whose transition semigroup
$(Q_t)_{t\geq0}$ is determined by
 \bgeqn\label{CH3LaplaceQ}
 \int_0^{\infty}e^{-zy}Q_t(x,dy)=\exp\{-x\psi_t(z)\},\quad
 t,x,z\geq0,
 \edeqn
where $\psi_t(z)$ is the unique solution of
 $$
\frac{\partial}{\partial t}\psi_t(z) = \frac{1}{1+\beta}\ga
\psi_t(z)^{1+\beta}, \qquad \psi_0(z) = z.
 $$
It is easy to find that
 $$
\psi_t(z)=z\l(\frac{1+\beta}{1+\beta+\ga\beta
tz^{\beta}}\r)^{1/\beta}.
 $$
In the sequel of the paper, we shall always assume $\ga>0$ unless
otherwise specified. Then for any $t>0$ we have
 \beqlb\label{CH3finite}
\lim_{z\rightarrow\infty}\psi_t(z)
 =
\left(\frac{1+\beta}{\gamma\beta t}\right)^{1/\beta}
 =:
\psi_t(\infty)<\infty.
 \eeqlb
Letting $z\rightarrow\infty$ in (\ref{CH3LaplaceQ}) yields
 \beqnn
Q_t(x,\{0\})
 =
\exp\{-x\psi_t(\infty)\}, \quad t>0,\, x\geq0.
 \eeqnn
From (\ref{CH3LaplaceQ}) it follows that
 $$
 Q_t(x_1+x_2,\cdot)=Q_t(x_1,\cdot)\ast Q_t(x_2,\cdot),\quad
 t,x_1,x_2\geq0.
 $$
In view of this infinite divisibility and (\ref{CH3finite}), we have
 \bgeqn\label{CH3defK}
\psi_t(z) = \int_0^{\infty}(1-e^{-zy})\kappa_t(dy), \quad t>0,
~z\geq0,
 \edeqn
for a family of finite diffuse measures $(\kappa_t)_{t>0}$ on
$(0,\infty)$; see, e.g., Bertoin and Le Gall \cite{BL00}.

A {\it coalescing Brownian flow} $\{\phi(a,t): a\in\mbb{R}, t\geq 0\}$ is by
definition an $\bR$-valued two-parameter process such that for every
$a\in\mbb{R}$, $t\mapsto \phi(a,t)$ is continuous; for every $t\ge 0$,
$a\mapsto \phi(a,t)$ is nondecreasing and right continuous; and for any $n\ge
1$ and $(a_1,\ldots,a_n)\in \mbb{R}^n$ the probability law of $(\phi(a_1,t),
\ldots, \phi(a_n,t))$ follows that of the {\it coalescing Brownian motion}
starting at $(a_1,\ldots,a_n)$; see \cite{Arr79} and \cite{FINR04} for more
details.

Let ${\cal M}$ be the space of Radon measures on $\mbb{R}$ endowed with the
topology of vague convergence. Let ${\cal M}_a$ be the subset of ${\cal M}$
consisting of purely atomic Radon measures. Let ${\cal B}_0$ be the space of
bounded Borel functions on $\mbb{R}$ with bounded supports. Suppose that
$\{(\phi_1(t),\phi_2(t),\ldots): t\ge0\}$ is a countable system of coalescing
Brownian motions and $\{(\xi_1(t),\xi_2(t),\ldots): t\ge0\}$ is a countable
system of independent CSBP's with $(1+\beta)$-stable branching law. We assume
{ that} the two systems are defined on a complete probability space and are
independent of each other. In addition, we assume { that}
$\{\phi_1(0),\phi_2(0),\ldots\}\cap [-l,l]$ is a finite set for every finite
$l\ge 1$. Then we define the ${\cal M}_a$-valued process
 \begin{equation}\label{def-m}
X_t = \sum_{i=1}^\infty \xi_i(t) \delta_{\phi_i(t)}, \qquad t\ge 0.
 \end{equation}
For any $t\ge 0$ let ${\cal G}_t = \sigma({\cal F}^\phi_t\cup {\cal
F}^\xi_t)$, where
 $$
{\cal F}^\phi_t := \sigma(\{\phi_i(s): 0\leq s\leq t;
i=1,2,\ldots\})
 $$
and
 $$
{\cal F}^\xi_t := \sigma(\{\xi_i(s): 0\leq s\leq t; i=1,2,\ldots\}).
 $$

 \begin{Theo}\label{CH3main0}
The process $\{X_t: t\geq0\}$ defined by (\ref{def-m}) is a right
continuous $({\cal G}_t)$-Markov process with transition semigroup
$(P_t)_{t\ge 0}$ given by
 \beqlb\label{laplace0}
\int_{{\cal M}_a}e^{-\la \nu, f\ra} P_t(\mu,d\nu)
 =
{\mbb{E}}\bigg[\exp\bigg\{-\int_{\mbb{R}}
\psi_t(f(\phi(a,t)))\mu(da)\bigg\}\bigg]
 \eeqlb
for $\mu\in {\cal M}_a$ and $f\in {\cal B}_0$, where $\phi(a,t)$ is
a coalescing Brownian flow.
\end{Theo}

\proof By the additivity of the CSBP's it is easy to see $\{X_t:
t\geq0\}$ a right continuous $({\cal G}_t)$-Markov process. By the
independence of the two systems $(\phi_1(t), \phi_2(t), \ldots)$ and
$(\xi_1(t), \xi_2(t), \ldots)$ we have
 \beqnn
{\mbb{E}}[\exp\{-\la X_t, f\ra\}]
 \ar=\ar
{\mbb{E}}\bigg[{\mbb{E}}\bigg[\exp\bigg\{-\sum_{i=1}^\infty
\xi_i(t)f(\phi_i(t)) \bigg\}\bigg|{\cal F}^\phi_t\bigg]\bigg] \cr
 \ar=\ar
{\mbb{E}}\bigg[\exp\bigg\{-\sum_{i=1}^\infty \xi_i(0)
\psi_t(f(\phi_i(t)))\bigg\}\bigg] \cr
 \ar=\ar
{\mbb{E}}\bigg[\exp\bigg\{-\sum_{i=1}^\infty \xi_i(0)
\psi_t(f(\phi(a_i,t)))\bigg\}\bigg] \cr
 \ar=\ar
{\mbb{E}}\bigg[\exp\bigg\{-\int_{\mbb{R}}X_0(da)
\psi_t(f(\phi(a,t)))\bigg\}\bigg],
 \eeqnn
where $a_i=\phi_i(0)$ for $i\geq1$ and $\phi(a,t)$ is a coalescing Brownian
flow independent of $(\xi_1(t),\xi_2(t),\ldots)$. A similar calculation shows
that $\{X_t: t\geq0\}$ has { the} transition semigroup $(P_t)_{t\ge 0}$ given
by (\ref{laplace0}).  \qed

{ \bdefinition By a {\it super-coalescing Brownian motion (SCBM)} we mean a
Markov process whose transition semigroup $(P_t)_{t\ge 0}$ is given by
(\ref{laplace0}) with $\mu\in {\cal M}$. \edefinition

From (\ref{laplace0}), it is easy to see that $P_t(\mu, {\cal M}_a)=1$ for $t>0$
and $\mu\in{\cal M}$. }Then the process constructed by (\ref{def-m}) is a
special case. We will give a formulation of the SCBM with an arbitrary
initial state $\mu\in {\cal M}$. To this end, let us review some basic facts
on CSBP's. Let $Q^{0}_t(x,\cdot)$ denote the restriction of the measure
$Q_t(x,\cdot)$ to $(0,\infty)$. Since the origin 0 is a trap for the CSBP,
the family of kernels $(Q^0_t)_{t\geq0}$ also constitutes a semigroup. Based
on (\ref{CH3LaplaceQ}) and (\ref{CH3defK}) one can check that
 \beqnn
 \ar\ar\int_0^{\infty}(1-e^{-zy})\kappa_{s+t}(dy)
 =\int^{\infty}_0\kappa_s(dx)\int_0^{\infty}(1-e^{-zy})Q^0_t(x,dy),
 \quad s,t>0,~z\geq0.
 \eeqnn
Then $\kappa_sQ^0_t=\kappa_{s+t}$. Therefore, $(\kappa_t)_{t>0}$ is
an entrance law for $(Q^0_t)_{t\geq0}$.

Let ${\bf W}$ be the set of right continuous nonnegative functions
on $(0,\infty)$ satisfying $w(t)=0$ for $t\geq\tau_0(w)$, where
$\tau_0(w):=\inf\{s>0: w(s)=0\}$. Let $0$ denote the path that is
constantly zero. Let $\mathfrak{B}({\bf W})$ be the natural
$\sigma$-algebra on ${\bf W}$ generated by the coordinate process
$\{w(s): s>0\}$. By the general theory of Markov processes, there
exits a unique $\sigma$-finite measure ${\bf Q}_{\kappa}$ on $({\bf
W}, \mathfrak{B}({\bf W}))$ such that ${\bf Q}_{\kappa}(\{0\}) = 0$
and
 \beqlb\label{CH3excursion}
 \ar\ar{\bf Q}_{\kappa}(w(t_1)\in dy_1,w(t_2)\in
dy_2,\ldots,w(t_n)\in dy_n)\cr
 \ar\ar\quad
=\kappa_{t_1}(dy_1)Q^0_{t_2-t_1}(y_1,dy_2)\ldots
Q^0_{t_n-t_{n-1}}(y_{n-1},dy_n)
 \eeqlb
for $0<t_1<t_2<\ldots<t_n$ and $y_1,y_2,\ldots,y_n\in (0,\infty)$;
see Proposition~3.5 of Getoor and Glover \cite{GeGl87}. By
Theorem~8.22 of Li \cite{Li11} we have $w(t)\to 0$ as $t\to 0$ for
${\bf Q}_{\kappa}$-almost every $w\in {\bf W}$. Then we can think of
${\bf Q}_{\kappa}$ as a $\sigma$-finite measure on the set ${\bf
W}_0$ of right continuous nonnegative functions on $[0,\infty)$
satisfying $w(0) = w(t)=0$ for $t\geq\tau_0(w)$, where
$\tau_0(w):=\inf\{s>0: w(s)=0\}$. The measure ${\bf Q}_\kappa$ is
known as the {\it excursion law} of the CSBP. Let $\mathfrak{B}({\bf
W}_0)$ and $\mathfrak{B}_t=\mathfrak{B}_t({\bf W}_0)$ denote the
natural $\sigma$-algebras on ${\bf W}_0$ generated by $\{w(s):s\geq
0\}$ and $\{w(s): 0\leq s\leq t\}$, respectively. For $r>0$, let
${\bf Q}_{\kappa,r}$ denote the restriction of ${\bf Q}_{\kappa}$ to
${\bf W}_r:=\{w\in{\bf W}_0:\tau_0(w)>r\}$. Note that
 $$
 {\bf Q}_{\kappa}({\bf W}_r)={\bf Q}_{\kappa,r}({\bf
 W}_r)=\kappa_r(0,\infty)=\psi_r(\infty)<\infty.
 $$
\begin{Lemma}
\label{CH3CSBP} The coordinate process $\{w(t+r): t\geq 0\}$ under
${\bf Q}_{\kappa,r}\{\cdot|\mathfrak{B}_r\}$ is a CSBP with
transition semigroup $(Q_t)_{t\geq 0}$. \end{Lemma}

\noindent\textit{Proof.} This follows from (\ref{CH3excursion}) by
standard arguments. We here give a detailed proof for the convenience
of the reader. For any $A\in \mathfrak{B}((0,\infty))$ we can use
(\ref{CH3excursion}) to see
 \beqnn
{\bf Q}_{\kappa,r}(w(r)\in A,w({r+t})=0)\ar=\ar
 {\bf Q}_{\kappa,r}(w(r)\in A)-{\bf Q}_{\kappa,r}(w(r)\in A,w({r+t})>0)\cr
 \ar=\ar\kappa_r(A)-\int_{A}\kappa_r(dx)Q^0_t(x,(0,\infty))\cr
 \ar=\ar\kappa_r(A)-\int_{A}\kappa_r(dx)(1-Q_t(x,\{0\}))\cr
 \ar=\ar\int_{A}\kappa_r(dx)Q_t(x,\{0\})
 \eeqnn
for $t>0$. Since 0 is a trap, we have
 \beqnn
\ar\ar{\bf Q}_{\kappa,r}(w(r)\in dy,w(t_1)\in dy_1,\ldots,w(t_n) \in
dy_n)\cr
 \ar\ar\quad
= \kappa_{r}(dy)Q_{t_1-r}(y,dy_1)\ldots Q_{t_n-t_{n-1}}
(y_{n-1},dy_n),
 \eeqnn
for $0<r<t_1<t_2<\ldots<t_n$ and $y>0$, $y_1,y_2,\ldots,y_n\in
[0,\infty)$. Thus for $A\in\mathfrak{B}((0,\infty))$ and
$A_1,\ldots,A_n\in \mathfrak{B}([0,\infty))$,
 \beqnn
\ar\ar\int_{\{w(r)\in A\}} {\bf Q}_{\kappa,r} \{w(t_1)\in
A_1,\ldots,w(t_n)\in A_n|\mathfrak{B}_r\}{\bf Q}_{\kappa,r}(dw)\cr
 \ar\ar\quad
=\int_{\{w(r)\in A\}}1\{w(t_1)\in A_1\ldots w(t_n)\in A_n\} {\bf
Q}_{\kappa,r}(dw)\cr
 \ar\ar\quad
=\int_A\int_{A_1}\,\ldots\,\int_{A_n}\kappa_{r}(dy)Q_{t_1-r}(y,dy_1)\ldots
 Q_{t_n-t_{n-1}}(y_{n-1},dy_n)\cr
 \ar\ar\quad
=\int_{\{w(r)\in A\}}\int_{A_1}\ldots\int_{A_n} Q_{t_1-r}(w_r,dy_1)
\ldots Q_{t_n-t_{n-1}}(y_{n-1},dy_n){\bf Q}_{\kappa,r}(dw).
 \eeqnn
Then an application of monotone class theorem yields the desired
result. \qed

We now consider an arbitrary initial measure $\mu\in {\cal M}$.
Suppose that on some complete probability space $(\Omega,\cal F,
\mbb{P})$  we have a coalescing Brownian flow $\{\phi(a,t):
a\in\mbb{R}, t\geq 0\}$ and a Poisson random measure $N(da,dw)$ on
$\mbb R\times{\bf W}_0$ with intensity measure $\mu(da){\bf
Q}_{\kappa}(dw)$. We assume that $\{\phi(a,t)\}$ and $\{N(da,dw)\}$
are independent of each other. Denote the support of $N$ by
$\textrm{supp}(N)=\{(a_i,w_i): i\ge1\}$. For $t\geq0$ let ${\cal
G}_t = \sigma({\cal F}^N_t\cup{\cal F}^\phi_t)$, where
 $$
{\cal F}^N_t := \sigma(\{w_i(s):0\leq s\leq t;i\ge 1\})
 $$
and
 $$
{\cal F}^\phi_t := \sigma(\{\phi(a,s): 0\leq s\leq t;a\in \mbb R\}).
 $$
Then we define the ${\cal M}$-valued process
 \bgeqn\label{CH3defX}
X_t=\int_\mbb{R}\int_{{\bf W}_0}w(t)\dz_{\phi(a,t)}N(da,dw), \quad
t>0,{\quad\text{ and }\quad X_0=\mu.}
 \edeqn

\begin{Theo}\label{CH3main}
{ The process $\{X_t: t\geq0\}$ defined by (\ref{CH3defX}) is an SCBM
starting from $\mu$ }and
 \beqlb\label{laplaceX}
{\mbb{E}}[\exp\{-\la X_{t}, f\ra\}]
 =
{\mbb{E}}\bigg[\exp\bigg\{-\int_{\mbb{R}}
\psi_t(f(\phi(a,t)))\mu(da)\bigg\}\bigg]
 \eeqlb
for $t>0$ and $f\in {\cal B}_0$.
\end{Theo}

\noindent\textit{Proof.} Note that for any $l\ge 1$ and $r>0$ we
have a.s. $m(l,r) := N([-l,l]\times {\bf W}_r)< \infty$. In fact, we
have
 $$
{\mbb{E}}[m(l,r)]
 =
\mu([-l,l]){\bf Q}_{\kappa}({\bf W}_r)
 =
\mu([-l,l]) \kappa_r(0,\infty)< \infty.
 $$
Then, since $\kappa_r(dx)$ is a diffuse measure, given ${\cal G}_r$
we can re-enumerate the set $\textrm{supp}(N)$ into
$\{(a_{k_i},w_{k_i}): i\ge1\}$ so that: (i) $|a_{k_1}|\le
|a_{k_2}|\le \ldots$; and (ii) $|a_{k_i}| = |a_{k_{i+1}}|$ implies
$w_{k_i}(r)< w_{k_{i+1}}(r)$. Note that this enumeration only uses
information from ${\cal G}_r$. As in the proof of Lemma 3.4 of
Dawson and Li \cite{DL03} one can see that $\{w_{k_i}(r+t):
t\geq0;i\ge 1\}$ under ${\mbb{P}}\{\cdot|{\cal G}_r\}$ are
independent CSBP's which are independent of $\{\phi(a,r+t): t\geq
0;a\in \mbb R\}$. Observe that
 \beqnn
X_{r+t} = \sum_{i=1}^\infty w_{k_i}(r+t) \dz_{\phi(a_{k_i},r+t)},
\qquad t\ge 0.
 \eeqnn
Then Theorem~\ref{CH3main0} implies that $\{X_{r+t}: t\geq0\}$ under
${{\mbb P}}\{\cdot|{\cal G}_r\}$ is a right continuous $({\cal
G}_{r+t})$-Markov process with transition semigroup $(P_t)_{t\ge
0}$. Thus $\{X_t: t>0\}$ under the non-conditioned probability
${\mbb{P}}$ is a right continuous $({\cal G}_t)$-Markov process with
transition semigroup $(P_t)_{t\ge 0}$. On the other hand, we have
 \beqnn
{\mbb{E}}[\exp\{-\la X_t, f\ra\}]
 \ar=\ar
{\mbb{E}}\bigg[\exp\bigg\{-\int_{\mbb{R}}\int_{{\bf W}_0}
w(t)f(\phi(a,t)) N(da,dw)\bigg\}\bigg] \cr
 \ar=\ar
{\mbb{E}}\bigg[\exp\bigg\{-\int_{\mbb{R}}\mu(da) \int_0^\infty
\Big(1-e^{-uf(\phi(a,t))}\Big) \kappa_t(du)\bigg\}\bigg] \cr
 \ar=\ar
{\mbb{E}}\bigg[\exp\bigg\{-\int_{\mbb{R}} \psi_t(f(\phi(a,t)))
\mu(da)\bigg\}\bigg].
 \eeqnn
That proves (\ref{laplaceX}) and {  also gives that  for each $t>0$, the
random measure $X_t$ has distribution $P_t(\mu,\cdot)$. We have thus
completed the proof.} \qed

\begin{Rem} Based on (\ref{laplaceX}) one can show $X_t\to \mu$ in probability
(in fact almost surely with a little more work) as $t\to 0$. From the above
construction we see that starting  from an arbitrary initial state in ${\cal
M}$, the SCBM collapses immediately into a purely atomic random measure with
a countable support. Then the masses located at different points evolve
according to independent CSBP's with the supporting points evolving according
to coalescing Brownian motions. { The construction (\ref{CH3defX}) of the
SCBM generalizes Theorem 3.5 of Dawson \textit{et al.} \cite{DLZ04}, where it
was assumed that the SCBM starts from a finite measure on $\mbb R$.}

\end{Rem}

We can also give a useful alternate characterization of the SCBM
following Zhou \cite{Zho07}. For $y_1\leq\ldots\leq y_{2n}\in
\bR^{2n}$, we write $(Y_1(t),\ldots,Y_{2n}(t))$ for a system of
coalescing Brownian motion starting at $(y_1,\ldots,y_{2n})$. Given
$\{a_1,\ldots,a_n\}\subset \bR^n$, throughout this paper we always
put
 \begin{equation}\label{def-h}
h_t(x) := \sum_{j=1}^n a_j1_{]Y_{2j-1}(t),Y_{2j}(t)]}(x), \qquad
t\ge 0, x\in \mbb{R}.
 \end{equation}
By applying Theorem \ref{CH3main} to $f=h_0$ we obtain

 \begin{Theo}\label{CH3main1}
Let $\{X_t: t\geq0\}$ be the SCBM defined by (\ref{CH3defX}). Then for any
$t>0$ we have
 \beqlb\label{laplaceX1}
{\mbb{E}}[\exp\{-\la X_{t},h_0\ra\}]
 =
{\mbb{E}}\bigg[\exp\bigg\{-\int_{\mbb{R}} \psi_t(h_t(a))
\mu(da)\bigg\}\bigg].
 \eeqlb
\end{Theo}

The above theorem shows that, the SCBM constructed in Zhou
\cite{Zho07} using approximation is actually the special case with
$\beta=1$ of the SCBM defined by (\ref{CH3defX}).

\section{An integral test on almost sure local extinction for SCBM}

Throughout this paper let $g(t), t>0$, be any nonnegative, nondecreasing and
right continuous function on $[0, \infty)$. Let $X$ be an SCBM. For such a
function $g$ we define the {\it extinction time} as
\[\tau:=\sup\{t\geq 0: X_t([-g(t),g(t)])\neq 0\}\]
with the convention $\sup\emptyset=0$.  Recall a standard result for
Brownian motion: If $\{B_t:t\geq0\}$ is a Brownian motion, then
\bgeqn\label{ineqnBM} {\bf P}\left\{\sup_{0\leq s\leq T}B_s\geq
C\right\}\leq \exp\left\{-\frac{C^2}{2T}\right\},\quad C\geq0,\,
T>0.
 \edeqn

\begin{Theo}\label{main-finite}
Assume { that} $X_0(dx)=dx$ and $\ga>0$. If
\begin{equation}\label{int}
\int_1^\infty g(y)y^{-1-1/\beta} dy<\infty,
\end{equation}
then
\begin{equation}\label{ext}
\bP\{\tau<\infty\}=1.
\end{equation}
\end{Theo}

\begin{proof} Without loss of generality we may assume that $X$ is defined as
(\ref{CH3defX}). We first  show that (\ref{ext}) holds given $g(t)\geq
t^{\de}$ for some constant $1/2<\de<1$ and for $t$ large enough. Put $t_n:=
e^{n}$. Set
$$
I_g^n:=[-2g(t_{n+1}), 2g(t_{n+1})].
$$
By the excursion representaion, we have that
$$
X_t=\int_{I_g^n}\int_{{\bf
W}_0}w(t)\delta_{\phi(a,t)}N(da,dw)+\int_{{\mbb R}\setminus
I_g^n}\int_{{\bf W}_0}w(t)\delta_{\phi(a,t)}N(da,dw)=:J_t^n+K_t^n.
$$
Note that on the event
$$
\Omega_{t_{n+1}}:=\left\{\sup_{0\leq t\leq
t_{n+1}}\phi(-2g(t_{n+1}),t)<
-g(t_{n+1})\right\}\bigcap\left\{\inf_{0\leq t\leq
t_{n+1}}\phi(2g(t_{n+1}),t)> g(t_{n+1})\right\},
$$
we have $$ X_t([-g(t_{n+1}),g(t_{n+1})]){=}
J_t^n([-g(t_{n+1}),g(t_{n+1})]),\quad { \forall\, }t\leq t_{n+1}.$$ Thus
 \beqnn
&&{\bf P}\{\exists t_n< t\leq
t_{n+1}, X_t([-g(t_{n+1}),g(t_{n+1})])\neq 0\}\\
 &&\qquad
\leq {\bf P}\{\exists t_n<
t\leq t_{n+1}, J_t^n([-g(t_{n+1}),g(t_{n+1})]) \neq 0\}+{\bf P}\{\Omega_{t_{n+1}}^c\}\\
 &&\qquad
= 1-{\bf P}\{J_t^n([-g(t_{n+1}),g(t_{n+1})])= 0, \forall\, t_n<
t\leq t_{n+1}\}+{\bf
P}\{\Omega_{t_{n+1}}^c\} \\
 &&\qquad
\leq 1-{\bf P}\{J_t^n({\mbb R})= 0, \forall\, t_n< t\leq t_{n+1}\}+{\bf
P}\{\Omega_{t_{n+1}}^c\},
 \eeqnn
 where the first inequality comes from the fact that $a\mapsto\phi(a,t)$ is
 non-decreasing.  By (\ref{ineqnBM}), we have
 $$
{\bf P}\{\Omega_{t_{n+1}}^c\}
 \leq
2\exp\left\{-\frac{g^2(t_{n+1})}{{ 2}t_{n+1}}\right\}
 \leq
2\exp\Big\{-\frac{t_{n+1}^{2\delta-1}}{{2}}\Big\}
 {\leq
2e^{-n}}
 $$
{ for sufficiently large $n\ge 1$. It follows that}
 $$
\sum_{n=1}^{\infty}{\bf P}\{\Omega_{t_{n+1}}^c\}<\infty.
 $$
On the other hand, the fact that $J_t^n(\mbb R)$ is a CSBP starting
from $4g(t_{n+1})$ yields
$${\bf P}\{J_t^n({\mbb R})= 0, \forall\, t_n< t\leq t_{n+1}\}={\bf
P}\{J_{t_n}^n(\mbb R)=0\}=\exp\left\{-4g(t_{n+1})\psi_{t_n}(\infty)
\right\}.$$ We have
$$
1-{\bf P}\{J_t^n({\mbb R})= 0, \forall\, t_n< t\leq t_{n+1}\}\leq
4g(t_{n+1})\psi_{t_n}(\infty).
$$
Moreover,
\begin{equation}\label{est2}
\begin{split}
\sum_{n=m}^\infty g(t_{n+1})\psi_{t_n}(\infty)
 &=
\sum_{n=m}^\infty
g(t_{n+1})\left(\frac{1+\beta}{\gamma\beta t_n}\right)^{1/\beta}\\
 &\leq
c(\ga,\beta)\int_{m+1}^\infty \frac{g(e^x)}{e^{(x-2)/\beta}} dx\\
 &\leq
e^{2/\beta}c(\ga,\beta)\int_{t_{m+1}}^\infty \frac{g(y)}{y^{1+1/\beta}}dy\\
 &<\infty,
\end{split}
\end{equation}
where $c(\ga,\beta) = \frac{1}{\beta}\left(\frac{1+\beta}
{\gamma\beta}\right)^{1/\beta}$.  Therefore, the desired result
follows from Borel-Cantelli lemma.

To show the desired result for any $g$ satisfying (\ref{int}), we
can consider function $g(t)+t^{\de}$ instead. It follows from the previous
result that (\ref{ext}) holds for function $g(t)+t^{\de}$. Then
plainly, it also holds for $g(t)$. \qed
\end{proof}


\begin{Theo}
\label{Main} Assume $X_0(dx)=dx$ and $\gamma>0$. If
$$
\int_1^{\infty}g(y)y^{-1-1/\beta}dy=\infty,
$$
then
$$
\mbb{P}\{\tau=\infty\}=1.
$$
\end{Theo}

\noindent\textit{Proof.} Delayed. \\ \bigskip

Before proceeding with the proof for Theorem \ref{Main} we first
define the nonnegative and strictly increasing sequences $(t_n),
(l_n)$ and $(r_n)$ as follows. Take $t_0\ge 0$ so that $g(t_0)\ge 1$
and define inductively
 $$
t_{n+1}:=\inf\{t\geq t_n: g(t)\geq 3g(t_n)\},\quad n\geq1.
 $$
 Then
 \beqlb\label{eqng}
g(t_n)\leq g(t_{n+1}-)\leq 3g(t_n)\leq g(t_{n+1}).
 \eeqlb
Define
 $$
r_n:=\frac{9}{10}g(t_n) \text{\,\, and \,\,} l_n:=\frac{31}{30}g(t_{n-1}).
 $$
 By (\ref{eqng}) we have
 \beqlb\label{intervaldiff}
r_n-l_n &=& \frac{3}{10}g(t_n)-\frac{3}{30}g(t_{n-1}) +
\frac{6}{10}g(t_n)-\frac{28}{30}g(t_{n-1}) \cr
 &\geq&
\frac{1}{10}\left(g(t_{n+1}-)-g(t_{n}-)\right)  +
\frac{6}{10}g(t_n)-\frac{28}{30}g(t_{n-1}) \cr
 &\geq&
\frac{1}{10}\left(g(t_{n+1}-)-g(t_{n}-)\right).
 \eeqlb
Let $I_n:=[-r_n, -l_n]\cup[l_n,r_n]$ and $I := \cup_{n=0}^{\infty} I_n$. We
need the following coalescing Brownian systems:
$$C_n=\{C_n(x,t); x\in I_n, t\geq0\}$$ and
$$C=\{C(x,t);x\in I, t\geq0\}$$ such that, for any finite set
$A\subset \mbb R$, both $\{C_n(x,\cdot); x\in A\cap I_n\}$ and
$\{C(x,\cdot); x\in A\cap I\}$ are coalescing Brownian motions
starting from $A\cap I_n$ and $A\cap I$, respectively.

Define
$$
\tau(x,y):=\inf\{t\geq0: C(x,t)=C(y,t)\}.
$$
According to the construction of coalescing Brownian motions on $\mbb R$ in
\cite{Arr79} (see also \cite{FINR04} for a more general model), we may
construct $\{C_n; n\geq1\}$ and $C$ from a countable family of independent
Brownian motions starting from $Q:=\{{i}/{2^k}; i,k\in \mbb Z\}$ such that
$C_n$ and $C_m$ are independent for  $n\neq m$ and
$$
C_n(x,t)=C(x,t)\quad \text{for } x\in I_n \text{ and }  t\leq\tau_n,
$$
where $\tau_n:=\tau(-r_n,-l_{n+1})\wedge \tau(-l_n,-r_{n-1})\wedge
\tau(r_{n-1},l_n)\wedge\tau(r_n,l_{n+1}).$

\bigskip

In the  sequel of this paper, { by $N(da,dw)$ we always denote} a Poisson
random measure on $\mbb R\times{\bf W}_0$ with intensity measure $da{\bf
Q}_{\kappa}(dw)$. For Lebesgue measure $L$ let $\{X_t^n:t\geq0\}$ and
$\{\bar{X}_t:t\geq0\}$ be defined by $X_0^n=1_{I_n}(x)dx$,
$\bar{X}_0=1_{I}(x)dx$ and for $t>0$,
 $$
X_t^n=\int_{I_n}\int_{{\bf W}_0}w(t)\dz_{C_n(a,t)}N(da,dw), \quad
\bar{X}_t=\int_{I}\int_{{\bf W}_0}w(t)\dz_{C(a,t)}N(da,dw).
 $$
Then by Theorem \ref{CH3main}, $\{X_t^n:t\geq0\}$  and $\{\bar{X}_t:t\geq0\}$
are SCBMs starting from $1_{I_n}(x)dx$ and $1_{I}(x)dx$, respectively. We
first prove the following lemma of key estimates. Theorem \ref{Main} will be
easily deduced from the lemma.

\begin{Lemma}
\label{KEYlemma1} Set $B_n:=[-g(t_n), g(t_n)]$. Assume that
$t^{1/2+\ez}\leq g(t)\leq 3^t$ for $t\geq1$ and some
$0<\ez<\frac{1}{2}$. Then we have
\begin{eqnarray}
\ar\ar\label{eqn1}\sum_{n=0}^\infty \mbb{P}\left\{X_{t_n}^n(B_n^c)>0\right\}<\infty;\\
\ar\ar\label{eqn2}\sum_{n=0}^\infty \mbb{P}\left\{X_{t_n}^n(\mbb R)>0\right\}=\infty; \\
 \ar\ar\label{eqn3}\sum_{n=0}^\infty \mbb{P}\left\{X_{t_n}^n(B_n)
>\bar{X}_{t_n}(B_{n})\right\}<\infty.
\end{eqnarray}
\end{Lemma}

\noindent\textit{Proof.} \textit{Proof for (\ref{eqn1}).} Set
$$
\Gamma_n:=\left\{\inf_{0\leq s\leq t_n} C_n(-r_n,s)\geq -g(t_n)
\right\}\bigcap\left\{ \sup_{0\leq s\leq t_n} C_n(r_n,s)\leq
g(t_n)\right\}.
$$
Then by (\ref{ineqnBM})
 \beqnn
\mbb{P}\{\Gamma_n^c\} &&\leq \mbb{P}\left\{\inf_{0\leq s\leq t_n}
C_n(-r_n,s)\leq -g(t_n)\right\}+{\mbb P}\left\{\sup_{0\leq s\leq
t_n} C_n(r_n,s)\geq
g(t_n)\right\} \\
 &&\leq
2\exp\left\{-\frac{g^2(t_n)}{ 200 t_n}\right\}\leq2
\exp\left\{-\frac{t_n^{2\ez}}{ 200}\right\}.
  \eeqnn
 By (\ref{eqng}) and the
assumption that $g(t)\leq3^t$ for $t>1$, we have
$$\sum_{n=1}^{\infty}\mbb{P}(\Gamma_n^c)<\infty.$$ Then (\ref{eqn1})
follows from
  \beqnn
\mbb{P}\left\{X_{t_n}^n(B_n^c)>0\right\}&=&{\mbb
P}\left\{X_{t_n}^n(B_n^c)>0;
\Gamma_n\right\}+\mbb{P}\left\{X_{t_n}^n(B_n^c)>0;\Gamma_n^c\right\}\\
&\leq& \mbb{P}\{\Gamma_n^c\}.
  \eeqnn

\textit{Proof for (\ref{eqn2})}. Since $X^n_{\cdot}(\mbb R)$ is a branching
process starting from $2(r_n-l_n)$,
 \beqnn
\mbb{P}\{X_{t_n}^n(\mbb R)>0\}&=&1-\exp\left\{-2(r_n-l_n)\psi_{t_n}(\infty)\right\}\\
 &\geq&
\min\left\{\frac{\psi_{t_n}(\infty)}{ 10}
\left(g(t_{n+1}-)-g(t_{n}-)\right),1-e^{-1}\right\},
 \eeqnn
where the inequality is deduced from (\ref{intervaldiff}) and the
elementary inequality $1-e^{-x}\geq  x/2$ for $0\leq x\leq 1$.
 Recall $c(\gamma,\beta)$ in
(\ref{est2}). Then (\ref{eqn2}) follows from
 \beqnn
 &&\sum_{n=1}^{\infty}\psi_{t_n}(\infty)\left(g(t_{n+1}-)-g(t_{n}-)\right)\\&&\quad=
 -\psi_{t_1}(\infty)g(t_1-)
 +\sum_{n=2}^{\infty}g(t_{n+1}-)(\psi_{t_n}(\infty)-\psi_{t_{n+1}}(\infty))\\
 &&\quad=-\frac{1}{t_1}g(t_1-)
 -\sum_{n=2}^{\infty}g(t_{n+1}-)\int_{t_n}^{t_{n+1}}
 \frac{\partial\psi_t(\infty)}{\partial t}dt\\
 &&\quad\geq-\frac{1}{t_1}g(t_1-)
 +c(\gamma,\beta)\sum_{n=2}^{\infty}\int_{t_n}^{t_{n+1}}g(t)
 t^{-1-1/\beta}dt\\
&&\quad=\infty.\eeqnn

\textit{Proof for (\ref{eqn3}).} Set
 $$
\Gamma_1(x,t,y):=\l\{\inf_{0\leq s\leq t}C(x,t)\geq
y\r\}\quad\text{and}\quad \Gamma_2(x,t,y):=\l\{\sup_{0\leq s\leq
t}C(x,t)\leq y\r\}.
 $$
Define
 \beqnn
\Omega_1:&=&\Gamma_1(-r_n,t_n,-g(t_n))\cap\Gamma_2(-l_{n},t_n,-g(t_{n-1}));\\
\Omega_2:&=&\Gamma_1(-r_{n-1},t_n,-g(t_{n-1}))\cap\Gamma_2(r_{n-1},t_n,
g(t_{n-1}));\\
\Omega_3:&=&\Gamma_1(l_n,t_n,g(t_{n-1}))\cap\Gamma_2(r_n,t_n,g(t_{n}));\\
\Omega_4:&=&\Gamma_1(l_{n+1},t_n,g(t_n))\cap
\Gamma_2(-l_{n+1},t_n,-g(t_n)).
 \eeqnn
Then again by (\ref{ineqnBM})
 \beqnn
\mbb{P}\left\{\cup\Omega_i^c\right\}&\leq&
4\exp\l\{-\frac{g^2(t_n)}{{ 1800}t_n}\r\}+4\exp\l\{-\frac{g^2(t_{n-1})}{{1800}t_n}\r\}\\
&\leq&4\exp\left\{-\frac{t_n^{2\ez}}{1800}\right\}+4\exp\left\{-\frac{t_n^{2\ez}}{200}\right\},
 \eeqnn
where the last inequality follows from (\ref{eqng}). Now, we are
ready to deduce (\ref{eqn3}). Note that on $\cap\Omega_i$,
$$
C_n(x,t)=C(x,t)\quad \text{for }t\leq t_n,\,x\in I_n.
$$
It follows that
 \beqnn
 &&\mbb{P}\left\{X_{t_n}^n(B_n) >\bar{X}_{t_n}(B_{n})\right\}\\
 &&\quad\leq
  \mbb{P}\left\{X_{t_n}^n(B_n)
  >\bar{X}_{t_n}(B_{n});\cap\Omega_i\right\}+\mbb{P}\{\cup\Omega_i^c\}\\
  &&\quad\leq 4\exp\left\{-\frac{t_n^{2\ez}}{1800}\right\}+4\exp\left\{-\frac{t_n^{2\ez}}{200}\right\}.
 \eeqnn
We thus obtain (\ref{eqn3}) by applying (\ref{eqng}) again and the
assumption that $g(t)\leq3^t$ for $t>1$. \qed

\medskip

\noindent\textit{Proof for Theorem \ref{Main}.}  { Firstly}, suppose
that $t^{1/2+\ez}\leq g(t)\leq 3^t$ for $t\geq1$ and some
$0<\ez<\frac{1}{2}$. Define
$$
\tau_0:=\sup\{t\geq0: \bar{X}_t\left([-g(t),g(t)]\right)>0\}.
$$
Note that $\{X_{\cdot}^n; n=1,2,\ldots\}$ are independent SCBMs. In
addition, the sequence $(t_n)$ defined after Theorem~\ref{Main} satisfies
$t_n\goto\infty$ since $g$ is increasing and $g(t)\leq e^t$ for all $t>0$. By
(\ref{eqn1}), (\ref{eqn2}) and the Borel-Cantelli lemma, a.s. all but a
finite number of the events $\{X_{t_n}^n(B_n^c)=0\} $ occur and an infinite
number of the events $\{X_{t_n}^n(\mbb R)>0\}$ occur. So a.s. an infinite
number of the events $\{X_{t_n}^n(B_n)>0\}$ occur. But by (\ref{eqn3}) and
the Borel-Cantelli lemma again, a.s. all but a finite number of the events
$\{X_{t_n}^n(B_n)\leq \bar{X}_{t_n}(B_n)\} $ occur. Thus an infinite number
of the events $\{\bar{X}_{t_n}(B_n)>0\}$ occur. This gives $$
\mbb{P}\{\tau_0=\infty\}=1.$$ {Let $X_t$ be defined as (\ref{CH3defX}) with
$X_0=dx$. Define $\{\tilde{X}_t:t\geq0\}$ with $\tilde{X}_0=1_{I}(x)dx$ and
 \beqnn
\tilde{X}_t:=\int_I\int_{{\bf W}_0}w(t)\dz_{\phi(a,t)}N(da,dw),\quad
t>0.
 \eeqnn
 Then $\tilde{X}$ is an SCBM starting from $1_I(x)dx$.} Obviously,
 $$
{X}_t\left([-g(t),g(t)]\right)>\tilde{X}_t\left([-g(t),g(t)]\right).
 $$
Then the fact that $\tilde{X}$ has the same distribution with
$\bar{X}$ yields
 $$
\mbb{P}\{\tau=\infty\}=1.
 $$
For more general $g$ satisfying $\int_1^{\infty}g(y)y^{-1-1/\beta}dy
=\infty$, we can consider function
 \[g_0(y):=(g(y)\wedge 3^y)\vee y^{\frac{1}{2}+\ez}.\]
First, one can check that
$$
\int_1^{\infty}{(g(y)\wedge 3^y)}y^{-1-1/\beta}dy=\infty.
$$
So,
\[\sup\{t\geq0: X_t([-g_0(t), g_0(t)])>0\}=\infty,\,\, \text{a.s.}\]
Meanwhile, according to Theorem \ref{main-finite},
 $$\sup\{t\geq0:
X_t([-t^{\frac{1}{2}+\ez}, t^{\frac{1}{2}+\ez}])>0\}<\infty \quad
\text{a.s.}$$ Then
$$\sup\{t\geq0: X_t([-g(t)\wedge 3^t, g(t)\wedge
3^t])>0\}=\infty\quad \text{a.s.}$$ Then the desired result follows
from $g(t)\geq g(t)\wedge 3^t$. We have thus finished the proof.
\qed

\bigskip
{\bf Acknowlegments}: The authors would like to thank an anonymous referee
for his/her remarks that improved significantly the presentation of the
paper.

\bigskip\bigskip

\noindent{\small Hui He and Zenghu Li: Laboratory of Mathematics and
Complex Systems, School of Mathematical Sciences, Beijing Normal
University, Beijing 100875, People's Republic of China. \\
\textit{E-mail:} {hehui@bnu.edu.cn} and {lizh@bnu.edu.cn}}

\bigskip

\noindent{\small Xiaowen Zhou: Department of Mathematics and
Statistics, Concordia University, 1455 de Maisonneuve Blvd. West,
Montreal, Quebec, H3G 1M8, Canada. \\
\textit{E-mail:} {xzhou@mathstat.concordia.ca}}

\end{document}